\documentclass[preprint,12pt]{elsarticle}

\usepackage{amssymb}
\usepackage{amsmath}
\usepackage{amsthm}
\usepackage{array}
\usepackage{rotating}
\newtheorem{thm}{Theorem}

\newtheorem{prop}{Proposition}
\newtheorem{lem}[thm]{Lemma}
\newdefinition{rem}{Remark}
\newdefinition{defi}{Definition}
\newproof{pf}{Proof}

\journal{*****}

\begin{document}

\begin{frontmatter}

\title{Meromorphically integrable homogeneous potentials with multiple Darboux points}
\author{Thierry COMBOT\fnref{label2}}
\ead{combot@imcce.fr}
\address{IMCCE, 77 Avenue Denfert Rochereau 75014 PARIS}

\title{Meromorphically integrable homogeneous potentials with multiple Darboux points}

\author{}

\address{}

\begin{abstract}
We prove that the only meromorphically integrable planar homogeneous potentials of degree $k\neq -2,0,2$ having a multiple Darboux point is the potential invariant by rotation. This case is a singular case of the Maciejewski-Przybylska relation on eigenvalues at Darboux points of homogeneous potentials, and needed before a case by case special analysis. The most striking application of this Theorem is the complete classification of integrable real analytic homogeneous potentials in the plane of negative degree.
\end{abstract}


\end{frontmatter}

\section{Introduction}

In this article, we want to study homogeneous potentials, which correspond to Hamiltonian systems of the form
$$H(p,q)=\frac{1}{2}\sum\limits_{i=1}^n p_i^2+V(q)$$
In particular, we are interested in meromorphic integrability of such potentials. Such type of problems have already been studied in several articles \cite{1,5,9,10,12,3,4}, in particular using the integrability conditions of the Morales-Ramis Theorem \cite{11}. Maciejewski-Przybylska found an important relation in \cite{4} which helps to test these integrability conditions. However, this relation does not hold in all cases, and not in particular in the case of a multiple Darboux point. The authors manage to circumvent this problem, analysing integrability of homogeneous polynomial potentials of degree $3,4$ in \cite{3,4}, but a more precise analysis of this special case seems possible. This case of a multiple Darboux point also appears in \cite{10}, and the analysis of this case leads to the proof that the only integrable real analytic homogeneous potential of degree $-1$ is the potential invariant by rotation.

Often, it is assumed that the potential should be meromorphic. As in our classification (Theorem \ref{thmmain1}), we will have to consider the potential $V=(q_1^2+q_2^2)^{k/2}$ which is not meromorphic for odd $k$, such an assumption is too restrictive. Still, in \cite{91}, the author present an extension of the Morales-Ramis-Simo Theorem dealing with such algebraic extensions. Following \cite{91}, we will consider $P_1,\dots,P_s\in\mathbb{C}[q_1,\dots,q_n,w_1,\dots,w_s]$ be $s$ polynomials, and $I$ the ideal generated by this polynomials, assumed to be a $n$-dimensional prime ideal. We now consider the algebraic complex manifold $\mathcal{S}=I^{-1}(0)$ and a open set $\Omega\subset \mathcal{S}$ such that
\begin{itemize}
\item The rank of the Jacobian of the application $w\mapsto (P_1(w),\dots,P_s(w))$ is maximal on $\mathcal{S}$ except at most on a codimension $1$ subset of $\mathcal{S}$, noted $\Sigma(\mathcal{S})$.
\item $\exists k_0,\dots,k_s \in\mathbb{Z},\; k_0\neq 0, \forall \alpha\in\mathbb{C}^*$
$$(q,w)\in\Omega \Rightarrow (\alpha^{k_0} q_1,\alpha^{k_0} q_2,\alpha^{k_1} w_1,\dots, \alpha^{k_s} w_s)\in\Omega$$
\item $\Omega \cap \Sigma(\mathcal{S})=\emptyset$
\end{itemize}
Now we can define on $\Omega$ a holomorphic function $V$, which will be our potential. We will call such potential a \emph{holomorphic homogeneous potential on $\Omega$}, and this will be the general setting for homogeneous potentials for the whole article. Remark that this case contains also the case of a meromorphic potential on an open set, just by removing the points of $\Omega$ on which $V$ is singular. Our interest is to study the integrability of such potentials. Let us first define what we will call meromorphically integrable

\begin{defi}
Let $V$ be a holomorphic homogeneous potential on $\Omega$. We say that $V$ is \emph{meromorphically integrable} if there exists $n$ first integrals, meromorphic on $\mathbb{C}^n\times \Omega$, functionally independent and in involution.
\end{defi}

The main theorems of this article are the following

\begin{thm}\label{thmmain1}
Let $V$ be a holomorphic homogeneous potential on $\Omega$ of degree $k\neq -2,0,2$ in dimension $2$. If $V$ has a multiple Darboux point $c\in \Omega$ and is meromorphically integrable, then
$$V=a(q_1^2+q_2^2)^{k/2}$$
\end{thm}

\begin{thm}\label{thmmain2}
Let $V$ be a real analytic homogeneous potential on $\mathbb{R}^2\setminus \{0\}$ of degree $k<0$. The potential $V$ is meromorphically integrable if and only if
$$k=-2 \hbox{ or } V=a(q_1^2+q_2^2)^{k/2} \;\;a\in\mathbb{R}$$
\end{thm}


Let us now remark that a homogeneous potential $V$, real analytic on $\mathbb{R}^2\setminus \{0\}$ is a holomorphic potential on an open set $\Omega\subset \mathbb{C}^2$. Using the homogeneity, we can moreover choose $\Omega$ such that $\forall \alpha\in\mathbb{C}^*,q\in\Omega,\;\; \alpha q\in\Omega$. Thus, in this Theorem, the meromorphic integrability of such a potential assumes here meromorphic first integrals on $\mathbb{C}^n\times \Omega$.\\

\noindent
The structure of the article is the following
\begin{itemize}
\item We first present and define the Morales-Ramis-Simo Theorem, its relation with Darboux points, and the motivation of the study of the particular case of multiple Darboux points.
\item We then prove that Theorem \ref{thmmain2} is a consequence of Theorem \ref{thmmain1}, as any real analytic integrable homogeneous potential will have a multiple real Darboux point.
\item We finally prove the Theorem \ref{thmmain1} using a notion of rigidity of higher variational equations, already presented in \cite{10}.
\end{itemize}

\section{Morales-Ramis-Simo Theorem and Darboux points}

\subsection{The Morales-Ramis-Simo Theorem}

Let us first write the Morales-Ramis-Simo Theorem, and in particular a version in \cite{91} for potentials on an algebraic variety.

\begin{thm}\label{thmmorales} (Combot \cite{91})
Let $V$ be a holomorphic homogeneous potential on $\Omega \subset \mathcal{S}$ and $\Gamma\subset \mathbb{C}^n\times \Omega$ a non-stationary orbit of $V$. If there are $n$ first integrals meromorphic on $\mathbb{C}^n\times \Omega$ of $V$ that are in involution and functionally independent over an open neighbourhood of $\Gamma$, then the identity component of Galois group of the variational equation near $\Gamma$ is Abelian over the base field of meromorphic functions on $\Gamma$.
\end{thm}

Remark that in the original statement, there is an additional hypothesis that $\Gamma \not\subset \mathbb{C}^n \times \Sigma(\mathcal{S})$. Here this hypothesis is automatically satisfied as we have already excluded $\Sigma(\mathcal{S})$ from $\Omega$. This Theorem needs an explicit algebraic orbit $\Gamma$, and in the case of homogeneous potentials, we can build generically such orbits thanks to Darboux points.

\begin{defi}\label{darb}
Let $V$ be a holomorphic homogeneous potential of degree $k$ on $\Omega \subset \mathcal{S}$. A \emph{Darboux point} $c\in \Omega\setminus \{0\}$ of $V$ satisfies the equation
$$\frac{\partial V}{\partial q_i }(c)=k c_i\quad i=1\dots n$$
To a Darboux point $c$, we associate an orbit (called \emph{homothetic orbit})
$$q(t)=\phi(t)^{k_0}.(c_1,\dots,c_n) \quad w(t)=(\phi(t)^{k_1} c_{n+1},\dots \phi(t)^{k_s} c_{n+s})$$
$$\frac{1} {2} k_0^2\phi^{2(k_0-1)}\dot{\phi}^2=\frac{\alpha}{k}\phi^{k_0k}+1$$
\end{defi}

Remark here that the derivation in $q_i$ on the algebraic variety $\mathcal{S}$ should be understand as the $w$ being algebraic (multivalued) functions (as presented in \cite{91}). Using Darboux points, Morales and Ramis proved a very effective criterion of meromorphic integrability

\begin{thm}\label{thmmorales2} (Morales-Ramis \cite{65}, Combot \cite{91} for the algebraic case)
Let $V$ be a homogeneous holomorphic potential on $\Omega \subset \mathcal{S}$ of homogeneity degree $k\in\mathbb{Z}^*$ and $c\in \Omega\setminus \{0\}$ a Darboux point of $V$. If $V$ is meromorphically integrable in the Liouville sense, then for any $\lambda\in \hbox{Sp}(\nabla^2 V(c))$, the couple $(k,\lambda)$ belongs to the table
\begin{center}\begin{tabular}{|c|c|c|c|}
\hline
$k$&$\lambda$&$k$&$\lambda$\\\hline
$\mathbb{Z}^*$&$\frac{1}{2}\,ik \left( ik+k-2 \right)$&$-3$&$-\frac{25}{8}+\frac{1}{8}(\frac{6}{5}+6 i)^2$ \\\hline
$\mathbb{Z}^*$&$\frac{1}{2}\left( ik+k-1 \right)  \left( ik+1 \right)$&$-3$&$-\frac{25}{8}+\frac{1}{8}(\frac{12}{5}+6 i)^2$ \\\hline
$2$&$\mathbb{C}$&$3$&$-\frac{1}{8}+\frac{1}{8}(2+6 i)^2$ \\\hline
$-2$&$\mathbb{C}$&$3$&$-\frac{1}{8}+\frac{1}{8}(\frac{3}{2}+6 i)^2$ \\\hline
$-5$&$-\frac{49}{8}+\frac{1}{8}(\frac{10}{3}+10 i)^2$&$3$&$-\frac{1}{8}+\frac{1}{8}(\frac{6}{5}+6 i)^2$ \\\hline
$-5$&$-\frac{49}{8}+\frac{1}{8}(4+10 i)^2$&$3$&$-\frac{1}{8}+\frac{1}{8}(\frac{12}{5}+6 i)^2$ \\\hline
$-4$&$-\frac{9}{2}+\frac{1}{2}(\frac{4}{3}+4i)^2$&$4$&$-\frac{1}{2}+\frac{1}{2}(\frac{4}{3}+4 i)^2$ \\\hline
$-3$&$-\frac{25}{8}+\frac{1}{8}(2+6 i)^2$&$5$&$-\frac{9}{8}+\frac{1}{8}(\frac{10}{3}+10 i)^2$ \\\hline
$-3$&$-\frac{25}{8}+\frac{1}{8}(\frac{3}{2}+6 i)^2$&$5$&$-\frac{9}{8}+\frac{1}{8}(4+6 i)^2$ \\\hline
\end{tabular}\end{center}
\end{thm}

The main difficulty in this Theorem is to compute the Darboux points, and then the eigenvalues at Darboux points. In particular, we obtain for each Darboux point of $V$ several conditions. So when studying families of homogeneous potentials, there will be for some particular cases fewer Darboux points than in the generic case. This happens for example when two Darboux points fuse together to give only one Darboux point.

\subsection{Multiple Darboux points}

As the Morales-Ramis Theorem gives necessary condition for integrability for each Darboux point, having fewer Darboux points increases the ``chances'' of satisfying these conditions. So, the multiple Darboux point case appears to be a special case which will lead to integrable potentials.

\begin{defi}\label{def1}
Let $V$ be a meromorphic homogeneous potential of degree $k$ on $\Omega \subset \mathcal{S}$. We consider the functions
$$G_i(q)=\frac{\partial V}{\partial q_i }(q)-k q_i \quad i=1\dots n$$
We say that $c$ is a \emph{multiple Darboux point} of $V$ if $G_i(c)=0,\;i=1\dots n$ and the rank of the Jacobian matrix of the application $q\mapsto (G_1(q),\dots,G_n(q))$ is not maximal at $c$.
\end{defi}

This is the classical definition of a multiple solution for the system of equations $(G_1(q),\dots,G_n(q))=0$. In this article, we are mainly interested by the application of the Morales-Ramis-Simo Theorem for multiple Darboux points. As we will see, assuming that the Darboux point is multiple has consequences on the possible eigenvalues at such a Darboux point.

\begin{prop}\label{prop1}
Let $c$ be a multiple Darboux point of a meromorphic homogeneous potential $V$. Then $k\in \hbox{Sp}(\nabla^2V(c))$.
\end{prop}

\begin{proof}
Using definition \ref{def1}, at a multiple Darboux point, the Jacobian matrix of the application $q\mapsto (G_1(q),\dots,G_n(q))$ has not a maximal rank. Computing this Jacobian matrix, it turns out to be $\nabla^2 V(c) -kI_n$. This matrix has not maximal rank, and thus has a non-zero vector $v$ in its kernel. Thus $v$ is an eigenvector of $\nabla^2 V(c)$ of eigenvalue $k$.
\end{proof}

So in particular, one of the integrability condition of Theorem  \ref{thmmorales} is automatically satisfied. In dimension $n$, this property on the spectrum is useful and important in the computations of higher variational equations, but still not enough a priori to completely classify such integrable potentials. Almost all non-integrability problems of homogeneous potentials is of the form\\

\noindent
\textbf{Problem}: Given a set $E$ of homogeneous potentials of degree $k$, find all elements in $E$ which are meromorphically integrable.
\bigskip

In this article, the set $E$ is the set of homogeneous potentials with a multiple Darboux point. As presented in \cite{10}, the difficulty to solve such a problem is closely related to the ``eigenvalue bounded'' property. Let us now note
$$\mathcal{M}=\left\lbrace V \hbox{ holomorphic on } \Omega \hbox{ homogeneous of degree } k \right\rbrace$$
Let $V\in\mathcal{M}$. We note $d(V)$ the set of Darboux points $c\in \Omega\setminus \{0\}$ of $V$. Given $c\in d(V)$, the spectrum of the Hessian matrix $\nabla^2 V(c)$ always contains the eigenvalue $k(k-1)$ because of the relation $\nabla^2 V(c)c=k(k-1)c$ (due to Euler relation for homogeneous functions of degree $k$). So we have $\hbox{Sp}\left(\nabla^2 V(c)\right)=\{k(k-1),\lambda_2,\dots,\lambda_n\}$ and we note
$$\Lambda(c)=\left\lbrace\begin{array}{c} \;\;\max(\lambda_2,\dots,\lambda_n)\quad\hbox{ if } \lambda_2,\dots,\lambda_n\in\mathbb{R} \\ -\infty\;\; \hbox{ otherwise}\\ \end{array}\right. $$

We consider $E\subset \mathcal{M}$ a subset of $\mathcal{M}$ and we define the following
$$\Lambda(E)= \sup\limits_{V\in E,\; d(V)\neq \varnothing} \inf\limits_{c \in d(V)} \Lambda(c)$$

In \cite{10}, the author proves a complete classification of homogeneous potentials of degree $-1$ in the plane assuming that $\Lambda <27$. The fact that $\Lambda(E)$ is finite is thus very important.

In the case of homogeneous potentials with multiple Darboux points, thanks to Proposition \ref{prop1}, we automatically obtain that in dimension $2$, $\Lambda(E)\leq k$ and thus is finite. In higher dimension, this property does not hold any more as there are other eigenvalues which could become arbitrary large.

Still there is an equivalent in higher dimension which still satisfy this eigenvalue bound. It corresponds to a more degenerate case, where the Jacobian matrix of Definition \ref{def1} is of rank $1$ (not only of rank $\leq n-1$). This case corresponds to the fusion of $n$ Darboux points in a generic terminal configuration.  As this case is less natural and more complicated, we will restrict ourselves for now in this article to the $2$-dimensional case.

\subsection{Motivations}

Let us explain why studying such a particular case is relevant. The first motivation is the application on real analytic potentials of negative degree. The second motivation is the Maciejewski-Pryzbylska relation \cite{4,25,97}. In the case of a planar rational homogeneous potential, the eigenvalues at a Darboux point $c$ are of the form $\{k(k-1),\lambda\}$. There exists generically a relation on eigenvalues (see \cite{97}) of the form
$$\sum\limits_{i=1}^p \frac{1}{\lambda_i-k}=c$$
where $c$ only depends on multiplicities of some roots of $V$ and $p$ is the number of Darboux points. Still this relation only holds when the Darboux points are simple. Indeed, if there is a multiple Darboux point, then $\lambda=k$ and this leads to a singularity in this relation. This relation is the key ingredient for classification of homogeneous potentials like in \cite{3,4,102}, and thus a difficult separate analysis of this case is necessary. The Theorem \ref{thmmain1} then allows to remove nicely this particular case. This always leads to the potential invariant by rotation.

\section{Theorem \ref{thmmain1} implies Theorem \ref{thmmain2}}

 Let us now prove that the Theorem \ref{thmmain2} is directly implied by Theorem \ref{thmmain1}.

\begin{proof}
A planar real analytic potential on $\mathbb{R}^2\setminus\{0\}$ can be written in polar coordinates $V=r^k U(\theta)$. A Darboux point $c=(r_0\cos \theta_0,r_0\sin\theta_0)$ of $V$ corresponds now to a critical point of $U$. The hypothesis $U(\mathbb{R})\subset \mathbb{R}$ implies that $U$ is $C^\infty$ on $\mathbb{R}$. The function $U$ is periodic, so there exists a minimum and a maximum for $U$. Assume first that $U$ is not constant. Then $\max U > \min U$. We have $3$ cases
\begin{itemize}
\item $\max U \geq \min U \geq 0$. Then we choose $\theta_0$ such that $U(\theta_0)= \max U$
\item $\max U \geq 0 \geq \min U$. If $\max U\neq 0$, we choose $\theta_0$ such that $U(\theta_0)= \max U$, otherwise we choose $\theta_0$ such that $U(\theta_0)= \min U$
\item $0 \geq \max U \geq \min U$. We choose then $\theta_0$ such that $U(\theta_0)= \min U$
\end{itemize}
Knowing that $\max U > \min U$, we get $U(\theta_0) \neq 0$. Then in all cases, we have
$$\frac{U''(\theta_0)}{U(\theta_0)} \leq 0$$
Knowing that $\theta_0$ is an extremum, we get
$$U(\theta_0)\neq 0 \quad U'(\theta_0)=0 \quad \frac{U''(\theta_0)}{U(\theta_0)} \leq 0$$

We define $c\in\mathbb{R}^2\setminus\{0\}$ by
$$c_1 = U(\theta_0)^{1/(2-k)} \cos\theta_0, \quad  c_2 = U(\theta_0)^{1/(2-k)} \sin\theta_0$$
After computation, we find that $c$ satisfies the equation
$$\partial_{q_1} V(c)=kc_1\qquad \partial_{q_2} V(c)=kc_2$$
So $c$ is a Darboux point of $V$. We now compute the eigenvalues of the Hessian $\nabla^2V(c)$, and we find
$$\hbox{Sp}(\nabla^2V(c))=\left\lbrace k(k-1),\frac{U''(\theta_0)}{U(\theta_0)}+k \right\rbrace $$
We have moreover that
$$\frac{U''(\theta_0)}{U(\theta_0)} \leq 0$$
and thus this second eigenvalue is less than $k$. If $V$ is meromorphically integrable than, thanks to Theorem~\ref{thmmorales}, the eigenvalues at Darboux points should belong to the Morales-Ramis table. Looking now precisely at it, we find out that when $k<0,\;k\neq -2$, the only possible eigenvalue allowed for meromorphic integrability less than $k$ is $k$ itself. So this implies that in fact
$$\hbox{Sp}(\nabla^2 V(c))=\{k(k-1),k\}$$
and $U''(\theta_0)=0$. This vanishing second order derivative implies (and in fact equivalent to) that $c$ is a multiple Darboux point. Theorem \ref{thmmain1} gives us that the only integrable homogeneous potential of degree $k$ with a multiple Darboux point is of the form $V=a(q_1^2+q_2^2)^{k/2}$. This potential (when $a\in\mathbb{R}$) is real analytic on $\mathbb{R}^2\setminus \{0\}$, and thus satisfy the hypothesis of Theorem \ref{thmmain2}. Thus
$$V=a(q_1^2+q_2^2)^{k/2},\;\; a\in\mathbb{R}$$
This case is effectively meromorphically integrable, as the angular momentum is a first integral of this potential. The case $k=-2$ is not analysed, but we do not need to. All planar homogeneous potentials of degree $-2$ are meromorphically integrable. This gives the Theorem.
\end{proof}

The proof is very similar to the one in \cite{10}, which proved the subcase $k=-1$ of Theorem \ref{thmmain2}, and also implies Theorem \ref{thmmain1} for $k=-1$.

\section{Proof of Theorem \ref{thmmain1}}

\subsection{Reduction by rotation-dilatation}

\begin{lem}
Let $V$ be a meromorphic homogeneous potential of degree $k\neq 0,2$ on $\Omega$ in dimension $2$. Assume there exists $c\in\Omega$ a non-degenerate Darboux point of $V$, with $c_1^2+c_2^2\neq 0$. Then after a rotation and dilatation, the potential $V$ has the following properties
\begin{itemize}
\item There exists a non-degenerate Darboux point of the form $c=(1,0,\dots)$.
\item We have $\hbox{Sp}(\nabla^2 V(c))=\{k(k-1),\lambda\}$, and the series expansion of $V$ at $c$ is of the form
$$V(c+q)=1+kq_1+\textstyle{\frac{1}{2}}k(k-1)q_1^2+\lambda q_2^2/2 +O(q^3)$$
\end{itemize}
\end{lem}

\begin{proof}
As there exists a non-degenerate Darboux point $c$, we can assume that $c=(\gamma,0,\dots)$ after a rotation (recall that $c_1^2+c_2^2\neq 0$). Then multiplying $V$ by a constant, we can put $\gamma=1$ (thanks to this hypothesis $k\neq 2$). Using Euler formula for homogeneous functions at $c$, we obtain that $V(c)=1$ and also $\partial_{q_1}V(c)=kV(c)$.

Differentiating the Euler relation and evaluating it at $(q_1,q_2)=c$, we also have
$$\partial_{q_1}V(c)+\partial_{q_1q_1}V(c)=k\partial_{q_1}V(c),\qquad \partial_{q_1q_2}V(c)+\partial_{q_2}V(c)=k\partial_{q_2}V(c) $$
Thus
$$\nabla^2 V(c) c= \left(\begin{array}{c}\partial_{q_1q_1}V(c)\\ \partial_{q_1q_2}V(c)  \end{array} \right)= \left(\begin{array}{c} (k-1)\partial_{q_1}V(c)\\ (k-1)\partial_{q_2}V(c)  \end{array} \right)=k(k-1)c$$
So the eigenvalue $k(k-1)$ always appear in the spectrum. So we can write $\hbox{Sp}(\nabla^2 V(c))=\{k(k-1),\lambda\}$. The series expansion of $V$ at $c$ follows.
\end{proof}

So, outside the case $c_1^2+c_2^2=0$, we can always assume that the Darboux point is of the form $c=(1,0,\dots)$. Let us now look more closely to what happens when $c_1^2+c_2^2=0$. The vector $c$ is still an eigenvector of $\nabla^2 V(c)$, with eigenvalue $k(k-1)$. Now according to \cite{72}, if $\nabla^2 V(c)$ is diagonalizable, then it has an orthonormal basis of eigenvectors. Thus the eigenvalue $k(k-1)$ is multiple in all cases (either $\nabla^2 V(c)$ is diagonalizable and then the eigenspace associated to eigenvalue $k(k-1)$ is of dimension $2$, or it is non-diagonalizable, and thus has a double eigenvalue). We conclude that the only possible spectrum in this case is $\{k(k-1),k(k-1)\}$, and thus this case cannot correspond to a multiple Darboux point due to Proposition \ref{prop1}.

So, if $c$ is a multiple Darboux point, we can always assume after rotation-dilatation that it is of the form $c=(1,0,\dots)$.

\subsection{Variational equations}

As we will see, the case $k=-1$ is a special case, and we will not analyze it here. This work has already been done by this author in \cite{10}. This article introduce in particular a non-degeneracy property, that we will use to prove Theorem \ref{thmmain1}. It happens that we are lucky, as the case $k=-1$ reveals to be the hardest case for $k<0,\; k\neq -2$. Indeed, we will find out that the non-degeneracy property will be satisfied for all $k$ except $k=-1$. The article \cite{10} has indeed required a more precise analysis of higher variational equations to conclude.

The first order variational equation can be written after a suitable variable change
$$\frac{1}{2} k^2\left( t^2-1\right) \ddot{X}+k\left(k- 1\right)t \dot{X}  -k X=0 $$
The solutions of this equation can be written
$$P_k=(t^2-1)^{1/k} \qquad Q_k=P_k \int (t^2-1)^{-(k+1)/k} dt$$
The function $Q_k$ is transcendental for $k\neq -2,-1,0,2$, and is in general an hyperelliptic integral. The case $k=-1$ is special (and more difficult) and has already been treated in \cite{10}.

We now follow the definition of higher variational equations given by Morales-Ramis-Simo \cite{2} page 860. Using their notation, the variational equations can be written
$$\dot{\varphi}_t^{(1)}=X_H^{(1)}\varphi_t^{(1)}$$
$$\dot{\varphi}_t^{(2)}=X_H^{(1)}\varphi_t^{(2)}+X_H^{(2)}(\varphi_t^{(1)})^2$$
$$\dot{\varphi}_t^{(3)}=X_H^{(1)}\varphi_t^{(3)}+2X_H^{(2)}(\varphi_t^{(1)},\varphi_t^{(2)})+X_H^{(3)}(\varphi_t^{(1)})^3$$
and they give a formula for any order $l$. In particular, at any order $l$, the last equation has always the following structure. There is a homogeneous part $\dot{\varphi}_t^{(l)}=X_H^{(1)}\varphi_t^{(l)}$, and non homogeneous terms involving functions already computed when solving lower order variational equations. So this last equation can be considered as a non homogeneous linear  equation.

We still assume that we are in the homogeneous potential case, with a Darboux point of the form $c=(1,0,\dots)$. The $X_H$ is the Hamiltonian field, and we may write $\varphi_t^{(l)}=(\dot{X}_1,\dot{X}_2,X_1,X_2)$. The $X_1$ correspond to a perturbation tangential to the homothetic orbit, and $X_2$ normal to this orbit. We see also that this variational equation is not linear. But for example at order $3$, instead of considering non linear terms like $(\varphi_t^{(1)})^3$, we replace it by solutions of the symmetric power of the equation satisfied by $\varphi_t^{(1)}$ (for this term, this gives the third symmetric power of the first order variational equation).

Computing variational equations up to order $l$ will produce monomials in the components of vectors $\varphi_t^{(1)},\dots,\varphi_t^{(l)}$. Equation $(13)$ of \cite{2} can be rewritten
$$\dot{\varphi}_t^{(l)}= \sum\limits_{j=1}^k \sum \frac{j!}{m_1! \dots m_s!} X_H^{(j)}((\varphi_t^{(i_1)})^{m_1},\dots,(\varphi_t^{(i_s)})^{m_s})$$
For each fixed $j$, the inner sum is a sum monomials of the form
\begin{equation}\label{mono}
(\varphi_t^{(1)})_{w_1}^{j_1}\dots (\varphi_t^{(l)})_{w_l}^{j_l}
\end{equation}
where $w$ indicates the component of vectors $\varphi_t$. Instead of computing $\varphi_t^{(i)}$, we compute directly these monomials. We note $y_{n_1,n_2,n_3,n_4}$ the sum over all monomials \eqref{mono} having exactly $n_1$ terms with $w=1$, $n_2$ terms with $w=2$, etc. Due to symmetries of higher variational equations, considering these $y_{n_1,n_2,n_3,n_4}$ are sufficient to analyze the variational equation (meaning that the derivatives of $y$ only involve $y$). This process has also linearized the variational equation as $y_{n_1,n_2,n_3,n_4}$ correspond to the monomials in the sum themselves. Building linear differential equations for the $y_{n_1,n_2,n_3,n_4}$ necessitates by the way to compute the symmetric product of differential systems (as done \cite{64}), as we need to build linear differential equation satisfied by monomials of the form \eqref{mono}. At order $l$, the variational equation now writes
\begin{equation}\label{nondeg}
\!\left(\begin{array}{c}\ddot{y}_{0,0,1,0}\\ \ddot{y}_{0,0,0,1} \end{array}\right)=\phi^{k_0(k-2)} \left(\begin{array}{c}k(k-1)y_{0,0,1,0}\\ \lambda y_{0,0,0,1} \end{array}\right) +\left(\begin{array}{c}
\sum \limits_{i=2}^{l} \phi^{k_0(k-1-i)}\sum\limits_{j=0}^i \frac{d_{i,j}}{(i-j)!j!} y_{0,0,i-j,j}\\ \sum \limits_{i=2}^{l} \phi^{k_0(k-1-i)}\sum\limits_{j=0}^i \frac{d_{i,j+1}}{(i-j)!j!} y_{0,0,i-j,j} \\ \end{array}\right) \!\!\!
\end{equation}
where $y_{i,0,j,0}$ satisfy differential equations corresponding to lower order variational equations. The coefficients $d_{i,j}$ are given by
$$d_{i,j}=\frac{\partial^{i+1}}{\partial q_1^{i-j+1} \partial q_2^{j}} V(c)$$

A visual process to build these differential systems is to see $y_{n_1,n_2,n_3,n_4}$ as $\dot{X_1}^{n_1}\dot{X_2}^{n_2} X_1^{n_3} X_2^{n_4}$. We differentiate this expression and simplify it using the relation
\begin{equation}\label{nondeg2}
\ddot{X}=\phi^{k_0(k-2)} \left(\begin{array}{cc}k(k-1)&0\\ 0&\lambda \end{array}\right)  X +\left(\begin{array}{c}
\sum \limits_{i=2}^{l} \phi^{k_0(k-1-i)}\sum\limits_{j=0}^i \frac{d_{i,j}}{(i-j)!j!} X_1^{i-j}X_2^{j}\\ \sum \limits_{i=2}^{l} \phi^{k_0(k-1-i)}\sum\limits_{j=0}^i \frac{d_{i,j+1}}{(i-j)!j!} X_1^{i-j}X_2^{j} \\ \end{array}\right)
\end{equation}
We suppress terms degree $>k$ that could appear, and then we formally replace back the $\dot{X_1}^{n_1}\dot{X_2}^{n_2} X_1^{n_3} X_2^{n_4}$ by $y_{n_1,n_2,n_3,n_4}$.

\begin{rem}\label{remhom}
Remark now that using the Euler relation for homogeneous function
$$q_1 \partial_{q_1}V+q_2 \partial_{q_2}V=kV$$
and differentiating it in $q_1$ or $q_2$ enough at $(q_1,q_2)=(1,0)$, we obtain the relations
$$\partial_{q_1^i q_2^j}V + \partial_{q_1^i q_2^j}V +\partial_{q_1^{i+1} q_2^j}V=k\partial_{q_1^i q_2^j}V,\quad i\geq 1,j\geq 0$$
This gives all derivatives $d_{k,j}$ of $V$ of order $l+1$ in function of lower order ones except $d_{l,l+1}$.
\end{rem}

By construction, the differential equations for the $y_{n_1,n_2,n_3,n_4}$ have special structure. In particular, the expression of $\dot{y}_{n_1,n_2,n_3,n_4}$ only involve terms with higher or equal sum of indexes. Thus, in particular, the differential equation for $y_{n_1,n_2,n_3,n_4},\;n_1+n_2+n_3+n_4=l$ is linear homogeneous and correspond to the $k$-th symmetric power of the first order variational equation. So the $y_{n_1,n_2,n_3,n_4},\;n_1+n_2+n_3+n_4=l$ are linear combinations of product of degree $l$ of solutions of the first order variational equation, which will be in our case after a variable change products of degree $l$ functions $P,Q$.

\subsection{Lemmas about monodromy}

Our main tool will be higher variational equation and the non degeneracy procedure of \cite{10}, in particular Theorem $8$. We still need to compute the monodromy of some functions appearing in higher variational equations. The monodromy analysis corresponds to the subsequent Lemmas.

\begin{lem}\label{lem1}
Let $\alpha,\beta\in\mathbb{Q}$ two rational numbers. Let $j\in\mathbb{Z}$ be an integer and $\gamma_j$ the closed path turning $j$ times around $1$ counter-clockwise and then $j$ times around $-1$ clockwise. If the function 
$$G=\int (t^2-1)^\alpha \int (t^2-1)^\beta dtdt$$
has a commutative monodromy, then the following matrix
$$A=\left(\begin{array}{cc}
\int\limits_{\gamma_{j_1}} (t^2-1)^\alpha dt&\int\limits_{\gamma_{j_2}} (t^2-1)^\alpha dt\\
\int\limits_{\gamma_{j_1}} (t^2-1)^\beta  dt&\int\limits_{\gamma_{j_2}} (t^2-1)^\beta  dt 
\end{array}\right) $$
has a zero determinant for any $j_1,j_2\in\mathbb{Z}$.
\end{lem}

Here the monodromy of the function $G$ should be understand has over the Riemann surface associated to $(t^2-1)^\beta,(t^2-1)^\alpha$. Indeed, these two functions are algebraic, thus define a Riemann surface $\mathcal{W}$ on which they are well defined and univalued. The path $\gamma_j$ we have chosen lives on this surface $\mathcal{W}$, and is a closed curve on this surface. Hence, the monodromy group is defined over the closed paths of $\mathcal{W}$.

Considering the Riemann surface $\mathcal{W}$ is necessary, as the integrability constraint on Galois group only concerns its identity component. By doing this construction, we ensure that the Galois group of $\mathbb{C}(t,G)$ over $\mathbb{C}(t,(t^2-1)^\beta,(t^2-1)^\alpha)$ is connected, and is the Zariski closure of the monodromy group we have defined above.

\begin{proof}
Let us note $\sigma_j$ the monodromy operator associated to the path $\gamma_j$. We first compute
$$\sigma_j\left( \int (t^2-1)^\alpha dt \right)= \int (t^2-1)^\alpha dt +\int\limits_{\gamma_j} (t^2-1)^\alpha dt$$
$$\sigma_j\left( \int (t^2-1)^\beta dt \right)= \int (t^2-1)^\beta dt +\int\limits_{\gamma_{j_1}} (t^2-1)^\beta dt$$
Let us now look at the action of these elements on $G$ for $j=j_1$ and $j=j_2$.
\begin{align*}
\sigma_{j_1}(G)'=\sigma_{j_1}\left((t^2-1)^\alpha \int (t^2-1)^\beta dt\right)=(t^2-1)^\alpha \left(\int (t^2-1)^\beta dt+A_{2,1}\right)\\
\Leftrightarrow\quad \exists K_1\in\mathbb{C} \quad \sigma_{j_1}(G)= G+ A_{2,1} \int (t^2-1)^\alpha dt +K_1
\end{align*}
For $j_2$, we get a similar formula
$$\exists K_2\in\mathbb{C} \quad \sigma_{j_2}(G)= G+ A_{2,2} \int (t^2-1)^\alpha dt +K_2$$
For the reciprocal, we get also the following formula
$$\sigma_{j_1}^{-1}(G)=G- A_{2,1} \left(\int (t^2-1)^\alpha dt-A_{1,1}\right) -K_1$$
$$\sigma_{j_2}^{-1}(G)=G- A_{2,2} \left(\int (t^2-1)^\alpha dt-A_{1,2}\right) -K_2$$
Let us now compute the action on $G$ of the commutator of these two monodromy elements
\begin{align*}
\sigma_{j_2}\sigma_{j_1}(G)=\\
G+ A_{2,2} \int (t^2-1)^\alpha dt+K_2+ A_{2,1} \left(\int (t^2-1)^\alpha dt+A_{1,2}\right) +K_1=\\
G+ (A_{2,2}+A_{2,1}) \int (t^2-1)^\alpha dt+K_2+ A_{2,1}A_{1,2} +K_1
\end{align*}
\begin{align*}
\sigma_{j_1}^{-1}\sigma_{j_2}\sigma_{j_1}(G)=G+A_{2,2} \left(\int (t^2-1)^\alpha dt-A_{1,1} \right)+K_2+ A_{2,1}A_{1,2}\\
\sigma_{j_2}^{-1}\sigma_{j_1}^{-1}\sigma_{j_2}\sigma_{j_1}(G)=G-A_{1,1}A_{2,2}+ A_{2,1}A_{1,2}
\end{align*}
So the commutator acts trivially if and only if the determinant of $A$ is zero.
\end{proof}

Let now make a more precise result

\begin{lem}\label{lem1b}
Let $\alpha,\beta\in\mathbb{Q}\setminus \mathbb{Z}$ two rational numbers. If the function 
$$G=\int (t^2-1)^\alpha \int (t^2-1)^\beta dtdt$$
has a commutative monodromy then $\alpha-\beta\in\mathbb{Z}$ or $-3/2-\alpha\in \mathbb{N}$ or $-3/2-\beta\in \mathbb{N}$.
\end{lem}

\begin{proof}
Let us first compute the determinant of matrix $A$ of the last Lemma. For $\alpha >-1$, we find that
$$\int\limits_{\gamma_j} (t^2-1)^\alpha dt=\left(1-e^{2ij\pi\alpha}\right)\int\limits_{-1}^1 (t^2-1)^\alpha dt=
\frac{(1-e^{2ij\pi\alpha})e^{i\pi\alpha}\Gamma(\alpha+1)\sqrt{\pi}} {\Gamma(\alpha+3/2)}$$
Let us now remark that the first expression is real analytic in $\alpha$. So, even if the second integral is only defined for $\alpha>-1$, we can use analytic continuation to obtain that the equality
$$\int\limits_{\gamma_j} (t^2-1)^\alpha dt= \frac{(1-e^{2ij\pi\alpha})e^{i\pi\alpha}\Gamma(\alpha+1)\sqrt{\pi}} {\Gamma(\alpha+3/2)}$$
is always valid for any $\alpha\in\mathbb{R}$. Remark still that the fact that $j$ is an integer is important, as it ensure that the right expression has no singularity. We can now compute the determinant of matrix $A$ which is
$$\frac{e^{i\pi(\alpha+\beta)}\Gamma(\alpha+1)\pi \Gamma(\beta+1)}{\Gamma(\alpha+3/2)\Gamma(\beta+3/2)} \times$$
$$(e^{2ij_2\pi\alpha}+e^{2ij_1\pi\beta}-e^{2ij_2\pi\beta}-e^{2ij_1\pi\alpha}+e^{2ij_1\pi\alpha+2ij_2\pi\beta}-e^{2ij_1\pi\beta+2ij_2\pi\alpha})$$
The $\Gamma$ terms in the denominator become singular if $\alpha$ or $\beta$ belongs to $\{-3/2,$ $-5/2,-7/2,\dots\}$. This is a special case of the Lemma. So the only interesting term is
$$(e^{2ij_2\pi\alpha}+e^{2ij_1\pi\beta}-e^{2ij_2\pi\beta}-e^{2ij_1\pi\alpha}+e^{2ij_1\pi\alpha+2ij_2\pi\beta}-e^{2ij_1\pi\beta+2ij_2\pi\alpha})$$
which becomes for $j_1=1,j_2=-1$
$$2i(-\sin(2\pi\beta)+\sin(2\pi\alpha)+\cos(2\pi\alpha)\sin(2\pi\beta)-\sin(2\pi\alpha)\cos(2\pi\beta))$$
Replacing the $\sin$ by $\cos$ and removing the square roots by multiplication with conjugates, we obtain that if this expression vanishes, then
$$16(\cos(\pi\alpha)^2-1)(\cos(\pi\beta)^2-1)(\cos(\pi\beta)^2-\cos(\pi\alpha)^2)=0$$
As $\alpha,\beta\notin \mathbb{Z}$, the only possibility left is that $\alpha-\beta\in\mathbb{Z}$.
\end{proof}

To prove Theorem \ref{thmmain1}, we will need to study the monodromy of some particular function.

\begin{lem}\label{lem2}
For $k\in\mathbb{Z}\setminus\{-2,-1,0,1,2\},\; l\in\mathbb{N},\;n\geq 0$ or $k=1,\; l\in\mathbb{N}^*$, the monodromy of
$$G_{l,k}(t)=\int \left(\int (t^2-1)^{-(k+1)/k} dt \right)^{l+1} (t^2-1)^{1/k} dt$$
is non commutative.
\end{lem}

\begin{proof}
We consider the differential field $K=\mathbb{C}(t,G_{l,k})$. Differentiating $G_{l,k}$, we have that $K$ contains the function
$$\left(\int (t^2-1)^{-(k+1)/k} dt \right)^{l+1}$$
This function is transcendental, so we can find a monodromy element $\mu$ and $\alpha\in\mathbb{C}^*$ such that
$$\mu\left(\int (t^2-1)^{-(k+1)/k} dt \right)= \int (t^2-1)^{-(k+1)/k} dt+ \alpha$$
So we get
\begin{equation}\label{eq5}
\mu^p(G_{l,k})-h(p)= \int \left(\int (t^2-1)^{-(k+1)/k} dt+p\alpha \right)^{l+1} (t^2-1)^{1/k} dt
\end{equation}
with $h(p)\in\mathbb{C}$. We also have that $\mu^p(G_{l,k})-h(p)$ is in $K$ for all $p\in\mathbb{Z}$. The right hand side of equation \eqref{eq5} is a polynomial in $p$, it is in $K$ for all $p$ so each coefficient in $p$ is in $K$. In particular, we get that
\begin{equation}\label{eq6}
\int \int (t^2-1)^{-(k+1)/k} dt (t^2-1)^{1/k} dt \in K
\end{equation}
We can now apply Lemma \ref{lem1b}. If the monodromy is commutative, then we are in one of the following cases
$$\frac{1}{k} \in\mathbb{Z},\;\; -\frac{k+1}{k} \in\mathbb{Z},\;\; \frac{k+1}{k}-\frac{3}{2} \in\mathbb{N},\;\; -\frac{3}{2}-\frac{1}{k} \in\mathbb{N},\;\; \frac{1}{k}+\frac{k+1}{k} \in\mathbb{Z}$$
None of these cases is possible if $\mid k\mid >2$. The case $k=1$ is a particular case. Indeed, the function \eqref{eq6} has a commutative monodromy. Still we also have using equation \eqref{eq5}
\begin{equation}\label{eq6bis}
\int \left(\int (t^2-1)^{-2} dt\right)^2 (t^2-1) dt \in K
\end{equation}
We can explicitly compute this equation, and we find a dilogarithmic term, and so the monodromy is non commutative.
\end{proof}

\subsection{The rigidity result}

Let us now prove Theorem \ref{thmmain1}. Our proof will be a non constructive one, meaning that we do not find integrable homogeneous potentials with multiple Darboux points. We will only prove that all of them are already known, and more precisely that there exists at most one such potential (for each homogeneity degree $k$).

\begin{lem}\label{lemmain}
Let $V_1,V_2$ be two integrable holomorphic homogeneous potentials on $\Omega$ of degree $k\neq -2,-1,0,2$ with a multiple Darboux point of the form $c=(1,0,\dots)\in\Omega$. Then $V_1=V_2$.
\end{lem}

\begin{proof}
To prove the Lemma, it is sufficient to prove that all derivatives of $V_1,v_2$ at $c$ are equal. Let us proof this result by recurrence.

As we assumed that $c$ is a Darboux point for both $V_1,V_2$, we have $V_1(c)=V_2(c)$, and their first derivatives are equal. The point $c$ is a multiple Darboux point for both potentials, so the eigenvalue $k$ belongs to the spectrum of both Hessian matrices of $V_1,V_2$. Their eigenvectors are equal (these are $c$ and a vector orthogonal to $c$), and thus the Hessian matrices are also equal. So the second derivatives are equal.

Let us now assume that all derivatives of $V_1,V_2$ are equal up to order $l\in\mathbb{N}, \; l\geq 3$. Let us now prove that the derivative of order $l+1$ are then also equal.

We first remark that using homogeneity, all derivatives of order $l+1$ of $V_1,V_2$ are equal except maybe $\partial_{q_2}^{l+1} V_1=d_{l,l+1}^{(1)}$, $\partial_{q_2}^{l+1} V_2=d_{l,l+1}^{(2)}$. This is here that we need to use the integrable hypothesis. The $l$-th variational equation (one of them) of $V_1,V_2$ is given respectively by
\begin{equation}\label{nondeg}\begin{split}
\ddot{X}=\phi(t)^{k_0(k-2)} \left(\begin{array}{cc}k(k-1)&0\\ 0&k \end{array}\right)  X +\left(\begin{array}{c}
\sum \limits_{i=2}^{n} \phi(t)^{k_0(k-1-i)}\sum\limits_{j=0}^i \frac{d_{i,j}^{(g)}}{(i-j)!j!} X_2^j X_1^{i-j}\\ \sum \limits_{i=2}^{n} \phi(t)^{k_0(k-1-i)}\sum\limits_{j=0}^i \frac{d_{i,j+1}^{(g)}}{(i-j)!j!} X_2^j X_1^{i-j}\\ \end{array}\right)\\
\end{split}\end{equation}
with $g=1,2$ respectively.

Let us now look at equation \eqref{nondeg}. This is a non homogeneous linear equation, so once we have found the expression of the non homogeneous terms, we can solve it using variation of parameters. Remark also that the highest order derivatives $d_{l,l+1}^{(g)}$ only appears in this equation and not in the lower order ones. We write the solution of the second equation of \eqref{nondeg} (after the variable change $k_0\dot{\phi}\phi^{k_0-1}/\sqrt{2} \longrightarrow t$)
\begin{equation}\label{solnon}
y^{(g)}(t)=y_{hom}^{(g)}(t)+y_{part_1}^{(g)}(t)+d_{l,l+1}^{(g)} y_{part_2}^{(g)}(t)
\end{equation}
isolating the term in $d_{l,l+1}^{(g)}$. The part $y_{hom}^{(g)}(t)$ is a solution of the homogeneous part, the solution $y_{part_1}^{(g)}(t)$ is a particular solution of equation \eqref{nondeg} without the term in $d_{l,l+1}^{(g)}$ and $y_{part_2}^{(g)}(t)$ is a particular solution of equation \eqref{nondeg} where all non homogeneous terms are removed except the one in $d_{k,k+1}^{(g)}$. Remark that the terms $y_{hom}^{(g)},y_{part_1}^{(g)}$ can be chosen equal for $g=1,2$ as the only difference between the variational equations of $V_1,V_2$ at order $l$ is $d_{l,l+1}^{(g)}$. Now let us look at $y_{part_2}^{(g)}$. It can be computed and one solution is
$$y_{part_2}^{(g)}(t)=\int \left(\int (t^2-1)^{-(k+1)/k} dt \right)^{l+1} (t^2-1)^{1/k} dt=G_{l,k}(t)$$
which is valid for both $g=1,2$. Thanks to Lemma \ref{lem2}, we know that the monodromy of this function $G_{l,k}$ is not commutative, and with Lemma \ref{lem1} that there exists a monodromy commutator $\sigma$ such that (with $\mu$ as in Lemma \ref{lem2})
\begin{equation}\begin{split}\label{eq7}
\sigma\mu^p(G_{l,k})=(p\alpha)^{l+1}\int (t^2-1)^{1/k} dt+\\
C_{l+1}^l (p\alpha)^l \left(\int \int (t^2-1)^{-(k+1)/k} dt (t^2-1)^{1/k} dt+\beta\right)  +\\
\sum\limits_{i=0}^{l-1} C_{l+1}^i (p\alpha)^i\sigma\left(\int \left(\int (t^2-1)^{-(k+1)/k} dt\right)^{n+1-i} (t^2-1)^{1/k} dt\right)
\end{split}\end{equation}
with $\beta$ a non-zero constant (as this $\beta$ corresponds to the determinant of Lemma \ref{lem2}) when $k\neq 1$. So applying this on $y^{(g)}$ gives
$$\sigma(y^{(g)})=\sigma(y_{hom}^{(g)})+\sigma(y_{part_1}^{(g)})+d_{l,l+1}^{(g)}\sigma(y_{part_2}^{(g)})$$
Remark that $\mu^p(y^{(g)})$ is also a solution of the second equation of \eqref{nondeg}. Now selecting the coefficient in $p^l$ of this equation, we get
\begin{equation}\begin{split}\label{eq8}
0=<p^l> \sigma\mu^p(y^{(g)})- \mu^p(y^{(g)})=<p^l> \sigma\mu^p(y^{(g)}_{hom})+\sigma\mu^p(y^{(g)}_{part_1}) -\\
\mu^p(y^{(g)}_{hom})-\mu^p(y^{(g)}_{part_1})+C_{l+1}^l (p\alpha)^l \beta d_{l,l+1}^{(g)}
\end{split}\end{equation}
So for $g=1,2$, we obtain $2$ affine functions in $d_{l,l+1}^{(1)},d_{l,l+1}^{(2)}$ respectively. As $y_{hom}^{(g)},y_{part_1}^{(g)}$ can be chosen equal for $g=1,2$, this is the same affine function, and as $\beta\neq 0$, their solutions are equal $d_{l,l+1}^{(1)}=d_{l,l+1}^{(2)}$.

In the special case $k=1$, we have $\beta=0$ in equation \eqref{eq7}, but the coefficient in $p^{l-1}$ of $\sigma\mu^p(G_{l,k})-G_{l,k}$ is non zero thanks to Lemma \ref{lem2}. So we just have to pick up the term in $p^{l-1}$ instead of $p^l$ in equation \eqref{eq8}. Thus the coefficient $p^{l-1}$ of equation \eqref{eq8} is an invertible affine function in $d_{l,l+1}^{(g)}$, and so the two solutions are equal $d_{l,l+1}^{(1)}=d_{l,l+1}^{(2)}$. 

Now to conclude, all the derivatives of $V_1,V_2$ coincide at $c$. This implies that $V_1=V_2$.
\end{proof}

Remark that in the hypotheses of this Lemma, the algebraic surfaces $\mathcal{S}$ and open sets $\Omega$ on which $V_1,V_2$ are defined could be different. The only thing we used is that $c$ is a multiple Darboux point for both of them, and that it lies in the open sets of definition for both $V_1,V_2$. Let us now prove Theorem \ref{thmmain1}.

\begin{proof}[Proof of Theorem \ref{thmmain2}]
After reduction by rotation dilatation, we can assume that a holomorphic potential $V_1$ on $\Omega$ with a multiple Darboux point $c$ is such that $c$ is of the form $c=(1,0,\dots)$. The potential
$$V_2=r^k\qquad \Omega=\{(q_1,q_2,r)\in\mathbb{C}^3,\;\; r^2=q_1^2+q_2^2,\; r\neq 0\}$$
is meromorphically integrable, and $c=(1,0,1)$ is a Darboux point. If $V_2$ is meromorphically integrable, than the hypotheses of Lemma \ref{lemmain} are satisfied, and thus $V_1=V_2$. The case $k=-1$ was already settle in \cite{10}.
\end{proof}

\section{Conclusion}

We have generalized Theorem $1$ of \cite{10} for any negative homogeneity degree $k<0$. We could of course ask what happen for positive homogeneity degree. Our result of Theorem \ref{thmmain1} still holds, but it is no longer sufficient to make a complete classification of real analytic integrable homogeneous potentials. Indeed, there are several admissible eigenvalues in the Morales-Ramis table lower than $k$, especially for $k=3,4,5$ which contains probably all exceptional cases of integrable potentials (see \cite{95}). Their complete classification seem to be possible but further work is required. In higher dimension, the multiple Darboux point analysis could also be interesting. The relation on eigenvalues is generalized in higher dimension in \cite{25}, and multiple Darboux points are still a problematic case for this relation. However, in higher dimension, there are many possible cases of multiplicity (depending on the rank of the Jacobian in Definition \ref{def1}), and also possibly continuums of Darboux points. Thus studying all cases for which this generalized relation does not hold seems uneasy. On the other hand, Theorem \ref{thmmain1} could probably be generalized in higher dimension: the maximum principle also leads to a multiple Darboux point, with a rank $1$ Jacobian in Definition \ref{def1}.

\bigskip

\label{}
\bibliographystyle{model1-num-names}
\bibliography{bibthese}

\begin{thebibliography}{17}
\expandafter\ifx\csname natexlab\endcsname\relax\def\natexlab#1{#1}\fi
\providecommand{\bibinfo}[2]{#2}
\ifx\xfnm\relax \def\xfnm[#1]{\unskip,\space#1}\fi
\bibitem[{Yoshida(1987)}]{1}
\bibinfo{author}{H.~Yoshida},
\newblock \bibinfo{title}{A criterion for the non-existence of an additional
  integral in {H}amiltonian systems with a homogeneous potential},
\newblock \bibinfo{journal}{Physica D: Nonlinear Phenomena}
  \bibinfo{volume}{29} (\bibinfo{year}{1987}) \bibinfo{pages}{128--142}.
\bibitem[{Morales-Ruiz and Simon(2009)}]{5}
\bibinfo{author}{J.~Morales-Ruiz}, \bibinfo{author}{S.~Simon},
\newblock \bibinfo{title}{On the meromorphic non-integrability of some
  {$N$}-body problems},
\newblock \bibinfo{journal}{Discrete and Continuous Dynamical Systems (DCDS-A)}
  \bibinfo{volume}{24} (\bibinfo{year}{2009}) \bibinfo{pages}{1225--1273}.
\bibitem[{Combot(2013)}]{9}
\bibinfo{author}{T.~Combot},
\newblock \bibinfo{title}{Integrability conditions at order $2$ for homogeneous
  potentials of degree $-1$},
\newblock \bibinfo{journal}{Non-linearity} \bibinfo{volume}{26}
  (\bibinfo{year}{2013}).
\bibitem[{Combot(2011)}]{10}
\bibinfo{author}{T.~Combot},
\newblock \bibinfo{title}{Generic classification of homogeneous potentials of
  degree $-1$ in the plane},
\newblock \bibinfo{journal}{arXiv:1110.6130}  (\bibinfo{year}{2011}).
\bibitem[{Morales-Ruiz and Ramis(2001)}]{12}
\bibinfo{author}{J.~Morales-Ruiz}, \bibinfo{author}{J.~Ramis},
\newblock \bibinfo{title}{Galoisian obstructions to integrability of
  hamiltonian systems ii},
\newblock \bibinfo{journal}{Methods and Applications of Analysis}
  \bibinfo{volume}{8} (\bibinfo{year}{2001}) \bibinfo{pages}{97--112}.
\bibitem[{Maciejewski and Przybylska(2004)}]{3}
\bibinfo{author}{A.~Maciejewski}, \bibinfo{author}{M.~Przybylska},
\newblock \bibinfo{title}{All meromorphically integrable {2D} {H}amiltonian
  systems with homogeneous potential of degree 3},
\newblock \bibinfo{journal}{Physics Letters A} \bibinfo{volume}{327}
  (\bibinfo{year}{2004}) \bibinfo{pages}{461--473}.
\bibitem[{Maciejewski and Przybylska(2005)}]{4}
\bibinfo{author}{A.~Maciejewski}, \bibinfo{author}{M.~Przybylska},
\newblock \bibinfo{title}{Darboux points and integrability of {H}amiltonian
  systems with homogeneous polynomial potential},
\newblock \bibinfo{journal}{Journal of Mathematical Physics}
  \bibinfo{volume}{46} (\bibinfo{year}{2005}) \bibinfo{pages}{062901}.
\bibitem[{Morales-Ruiz and Ramis(2001)}]{11}
\bibinfo{author}{J.~Morales-Ruiz}, \bibinfo{author}{J.~Ramis},
\newblock \bibinfo{title}{Galoisian obstructions to integrability of
  hamiltonian systems},
\newblock \bibinfo{journal}{Methods and Applications of Analysis}
  \bibinfo{volume}{8} (\bibinfo{year}{2001}) \bibinfo{pages}{33--96}.
\bibitem[{Combot(2012)}]{91}
\bibinfo{author}{T.~Combot},
\newblock \bibinfo{title}{A note on algebraic potentials and {M}orales-{R}amis
  theory},
\newblock \bibinfo{journal}{arXiv:1209.4747}  (\bibinfo{year}{2012}).
\bibitem[{Morales-Ruiz and Ramis(2001)}]{65}
\bibinfo{author}{J.~Morales-Ruiz}, \bibinfo{author}{J.~Ramis},
\newblock \bibinfo{title}{A note on the non-integrability of some hamiltonian
  systems with a homogeneous potential},
\newblock \bibinfo{journal}{Methods and applications of analysis}
  \bibinfo{volume}{8} (\bibinfo{year}{2001}) \bibinfo{pages}{113--120}.
\bibitem[{Przybylska(2007)}]{25}
\bibinfo{author}{M.~Przybylska},
\newblock \bibinfo{title}{Finiteness of integrable n-dimensional homogeneous
  polynomial potentials},
\newblock \bibinfo{journal}{Physics Letters A} \bibinfo{volume}{369}
  (\bibinfo{year}{2007}) \bibinfo{pages}{180--187}.
\bibitem[{Michal~Studzinski(2012)}]{97}
\bibinfo{author}{M.~P. Michal~Studzinski},
\newblock \bibinfo{title}{Darboux points and integrability analysis of
  hamiltonian systems with homogeneous rational potentials},
\newblock \bibinfo{journal}{arXiv:1205.4395}  (\bibinfo{year}{2012}).
\bibitem[{Llibre et~al.(2011)Llibre, Mahdi, and Valls}]{102}
\bibinfo{author}{J.~Llibre}, \bibinfo{author}{A.~Mahdi},
  \bibinfo{author}{C.~Valls},
\newblock \bibinfo{title}{Polynomial integrability of the {H}amiltonian systems
  with homogeneous potential of degree {$-3$}},
\newblock \bibinfo{journal}{Phys. D} \bibinfo{volume}{240}
  (\bibinfo{year}{2011}) \bibinfo{pages}{1928--1935}.
\bibitem[{Craven(1969)}]{72}
\bibinfo{author}{B.~Craven},
\newblock \bibinfo{title}{Complex symmetric matrices},
\newblock \bibinfo{journal}{Journal of the Australian Mathematical Society}
  \bibinfo{volume}{10} (\bibinfo{year}{1969}) \bibinfo{pages}{341--354}.
\bibitem[{Morales-Ruiz et~al.(2007)Morales-Ruiz, Ramis, and Sim\'o}]{2}
\bibinfo{author}{J.~Morales-Ruiz}, \bibinfo{author}{J.~Ramis},
  \bibinfo{author}{C.~Sim\'o},
\newblock \bibinfo{title}{Integrability of {H}amiltonian systems and
  differential {G}alois groups of higher variational equations},
\newblock \bibinfo{journal}{Annales scientifiques de l'Ecole normale
  sup{\'e}rieure} \bibinfo{volume}{40} (\bibinfo{year}{2007})
  \bibinfo{pages}{845--884}.
\bibitem[{Van~der Put and Singer(2003)}]{64}
\bibinfo{author}{M.~Van~der Put}, \bibinfo{author}{M.~Singer},
  \bibinfo{title}{Galois theory of linear differential equations}, volume
  \bibinfo{volume}{328}, \bibinfo{publisher}{Springer Verlag},
  \bibinfo{year}{2003}.
\bibitem[{Nakagawa and Yoshida(2001)}]{95}
\bibinfo{author}{K.~Nakagawa}, \bibinfo{author}{H.~Yoshida},
\newblock \bibinfo{title}{A list of all integrable two-dimensional homogeneous
  polynomial potentials with a polynomial integral of order at most four in the
  momenta},
\newblock \bibinfo{journal}{Journal of Physics A: Mathematical and General}
  \bibinfo{volume}{34} (\bibinfo{year}{2001}) \bibinfo{pages}{8611}.

\end{thebibliography}

\end{document}